\documentclass[11pt,reqno]{amsart}

\usepackage{amsthm,amsmath,amssymb}
\usepackage{graphicx}
\usepackage{float}
\usepackage[colorlinks=true,
linkcolor=blue,
urlcolor=blue,
citecolor=blue]{hyperref}
\usepackage{mathtools}
\usepackage{relsize}
\usepackage{ytableau}
\usepackage{tikz}
 \usepackage{enumerate}
\usepackage[toc,page]{appendix}
\usepackage{lscape}
\usepackage{multirow}
\usepackage{adjustbox,expl3,etoolbox} 
\usepackage[margin = 3.5cm]{geometry}
\usepackage{longtable}

\newtheorem{thm}{Theorem}[section]
\newtheorem{lem}[thm]{Lemma}
\theoremstyle{definition}

\newtheorem{prop}[thm]{Proposition}
\newtheorem{conj}[thm]{Conjecture}

\DeclareMathOperator\sym{Sym}
\DeclareMathOperator\alt{Alt}

\DeclareMathOperator{\cay}{Cay}

\letcs\replicate{prg_replicate:nn}

\title[]{On the second eigenvalue of a Cayley graph of the Symmetric group}

{\author[R. Maleki]{Roghayeh Maleki\textsuperscript{1}}
	\thanks{\textsuperscript{1} Department of Mathematics and Statistics, University of Regina}
	\address{Department of Mathematics and Statistics, University of Regina, Regina, Saskatchewan S4S 0A2, Canada}}\email{rmaleki@uregina.ca}
{\author[A. S. Razafimahatratra]{Andriaherimanana Sarobidy Razafimahatratra\textsuperscript{1,*}}
	\thanks{\textsuperscript{*} Corresponding author}
	\address{Department of Mathematics and Statistics, University of Regina,
		Regina, Saskatchewan S4S 0A2, Canada}\email{sarobidy@phystech.edu}}

\begin{document}	



\date{\today}

\maketitle

\begin{abstract}
	In 2020, Siemons and Zalesski [On the second eigenvalue of some Cayley graphs of the symmetric
	group. {\it arXiv preprint arXiv:2012.12460}, 2020] determined the second eigenvalue of the Cayley graph $\Gamma_{n,k} = \cay(\sym(n), C(n,k))$ for $k = 0$ and $k=1$, where $C(n,k)$ is the conjugacy class of $(n-k)$-cycles. In this paper, it is proved that for any $n\geq 3$ and $k\in \mathbb{N}$ relatively small compared to $n$,  the second eigenvalue of $\Gamma_{n,k}$  is the eigenvalue afforded by the irreducible character of $\sym(n)$ that corresponds to the partition $[n-1,1]$. As a byproduct of our method, the result of Siemons and Zalesski when $k \in \{0,1\}$ is retrieved. Moreover, we prove that the second eigenvalue of $\Gamma_{n,n-5}$ is also equal to the eigenvalue afforded by the irreducible character of the partition $[n-1,1]$.

\end{abstract}

\section{Introduction}

The second eigenvalue of a graph is a well-studied parameter in spectral graph theory (see \cite{cvetkovic1995second,neumaier1982second,nilli1991second} for instance). Given a graph $X = (V,E)$, the eigenvalues of $X$ (i.e., the eigenvalues of its adjacency matrix) are real and can be arranged in the following way $\lambda_1(X) \geq \lambda_2(X)\geq \ldots \geq \lambda_{|V(X)|}(X)$. When $X$ is a regular graph, the largest eigenvalue $\lambda_1(X)$ is equal to the valency of $X$. This paper is concerned with the study of the second (largest) eigenvalue, $\lambda_2(X)$, of a Cayley graph of the symmetric group. 

Let $G$ be a finite group with identity $e$, and let $C$ be a subset of $G\setminus\{e\}$ with the property that if $x\in C$ then $x^{-1} \in C$. The \emph{Cayley graph} $\cay(G,C)$ is the simple and undirected graph with vertex set $G$ and two group elements $g$ and $h$ are adjacent if and only if $hg^{-1} \in C$. We say that the graph $\cay(G,C)$ is a \emph{normal Cayley graph} if $gCg^{-1} = C$, for any $g\in G$. In other words, $\cay(G,C)$ is a normal Cayley graph if and only if $C$ is a union of conjugacy classes of $G$. The graph $\cay(G,C)$ is regular with valency equal to $|C|$. Since the elements of the left-regular representation $\mathrm{R}(G) = \left\{ \rho_g : G\to G,\ \rho_g(x) = gx \mid g\in G \right\}$ of $G$ act as automorphisms of $\cay(G,C)$, the latter is also vertex-transitive.

Let $n$ and $k$ be two integers such that $0\leq k \leq n-1$. We denote the conjugacy class of the $(n-k)$-cycles of $\sym(n)$ by $C(n,k)$. That is,
\begin{align*}
	C(n,k) = \left\{ \sigma (1,2,3,\ldots,n-k) \sigma^{-1} \mid \sigma \in \sym(n) \right\}.
\end{align*}
Let $\Gamma_{n,k}$ be the Cayley graph of $\sym(n)$ with connection set equal to $C(n,k)$. In other words, $\Gamma_{n,k}: = \cay(\sym(n),C(n,k))$. The graph $\Gamma_{n,k}$ is vertex-transitive and regular of valency $|C(n,k)| = \binom{n}{k}(n-k-1)!$. Since $C(n,k)$ is a conjugacy class, $\Gamma_{n,k}$ is  also a \emph{normal Cayley graph}. 

A famous result of Babai \cite{babai1979spectra} links the eigenvalues of a normal Cayley graph $\cay(G,C)$ to the irreducible characters of $G$. In particular, each irreducible character of $G$ determines an eigenvalue of $\cay(G,C)$ and the spectrum of $\cay(G,C)$ (the multiset consisting of all the eigenvalues) can be determined with the knowledge of the set of all irreducible characters $\operatorname{Irr}(G)$ of $G$. 

The following theorem was recently proved by Siemons and Zalesski \cite{siemons2020second}.

\begin{thm}
	For $n\geq 5$, the second eigenvalue of $\Gamma_{n,k}$, for $k\in \{0,1\}$, is as follows.
	\begin{itemize}
		\item If $k = 0$, then $\lambda_2(\Gamma_{n,k}) = (n-2)! \mbox{ when $n$ is even}$ and $\lambda_2(\Gamma_{n,k}) = 2(n-3)!$ when $n$ is odd.
		\item If $k = 1$, then $\lambda_2(\Gamma_{n,k}) = 3(n-3)(n-5)! \mbox{ when $n$ is even}$ and $\lambda_2(\Gamma_{n,k}) = 2(n-2)(n-4)!$ when $n$ is odd.
	\end{itemize}\label{thm:large}
\end{thm}

Finding the second eigenvalue of $\Gamma_{n,k}$, for arbitrary $k\leq n$, is an open problem posed by Siemons and Zalesski in \cite{siemons2020second}. In this paper, we find $\lambda_2\left(\Gamma_{n,k}\right)$ when $k$ is relatively small. Our main result is the following.
\begin{thm}
	For any $n\geq 3$ and $k\in \mathbb{N}$ such that $ 2\leq  k \leq \min \left(n,2\log_{\frac{k}{e}} \left(\frac{n(n-2)}{2e}\right)-1\right)$, the second eigenvalue of $\Gamma_{n,k}$ is 
	\begin{align*}
	\lambda_2\left(\Gamma_{n,k} \right) = \frac{k-1}{n-1}\binom{n}{k}(n-k-1)!.
	\end{align*}
	Moreover,  $\lambda_2\left(\Gamma_{n,k} \right)$ is afforded by the irreducible character of $\sym(n)$ that corresponds to the partition $[n-1,1]$.\label{thm:main}
\end{thm}

It is well-known that a Cayley graph $\cay(G,C)$ is connected if and only if $\langle C\rangle = G$. Note that if the graph $\Gamma_{n,k}$ is connected then, $C(n,k)$ has to contain odd permutations. Indeed, if $C(n,k)$ only contains even permutations, then due to the fact that the product of even permutations is an even permutation, we have $\langle C(n,k)\rangle \leq \alt(n)$, where $\alt(n)$ is the alternating group on $\{1,2,\ldots,n\}$. In fact, one can easily check that the set $\{(1,2,x) \mid x\in \{3,4,\ldots,n\} \}$ is contained in $\langle C(n,k)\rangle$, which implies that $\langle C(n,k)\rangle = \alt(n)$. In general, if $X = \cay(\sym(n),C)$ is a normal Cayley graph, then it is straightforward that $\langle C\rangle \triangleleft \sym(n)$. Since $\alt(n)$ is the only minimal normal subgroup of $\sym(n)$, for $n\geq 5$, we have $\langle C \rangle \in \{ \alt(n),\sym(n) \}$.

Consequently, if $n-k$ is odd, then the graph $\Gamma_{n,k}$ is disconnected and has two components. Since $\Gamma_{n,k}$ is vertex-transitive, these two components are isomorphic. Hence, the second eigenvalue of $\Gamma_{n,k}$ is equal to the second eigenvalue of $\cay(\alt(n),C(n,k))$, when $n-k$ is odd. The following result was proved by Huang and Huang \cite{huang2019second}.
\begin{thm}
	For any $n\geq 5$, the second eigenvalue of $\cay(\alt(n),C(n,n-3))$ is $\frac{n(n-2)(n-4)}{3}$.
\end{thm}

Therefore, the second eigenvalue of $\Gamma_{n,n-3}$ is $\frac{n(n-2)(n-4)}{3}$. It can be verified that this eigenvalue is also the one afforded by irreducible character corresponding to the partition $[n-1,1]$.

Other results on the second eigenvalue of $\Gamma_{n,k}$, when $k$ is large are also known. For instance, Diaconis and Shahshahani \cite{diaconis1981generating} proved that $\lambda_2(\Gamma_{n,n-2}) = \binom{n}{2} - n$, which coincides with our result in Theorem~\ref{thm:main} when $k = n-2$. This eigenvalue corresponds to the irreducible character of the partition $[n-1,1]$. Moreover, Huang et al. \cite{Cioaba2019second} proved that $\lambda_2(\Gamma_{n,n-4}) = \frac{n(n-2)(n-3)(n-5)}{4}$, which is again the eigenvalue corresponding to the partition $[n-1,1]$. We believe that our main result,  Theorem~\ref{thm:main}, also holds in the general case when $2\leq k\leq n-2$. For $n\in \{3,4,5,6,\ldots,21\}$ and  $2\leq k\leq n-2$, an exhaustive search on  \verb|Sagemath| \cite{sagemath} for  the second eigenvalue of $\Gamma_{n,k}$, showed that $\lambda_2(\Gamma_{n,k}) =  \frac{k-1}{n-1}\binom{n}{k}(n-k-1)!$. Due to these observations, we make the following conjecture.
\begin{conj}
	For any $n\in \mathbb{N}$ and $2\leq k \leq n-2$, the second eigenvalue of $\Gamma_{n,k}$ is given by the irreducible character corresponding to $[n-1,1]$, and its value is
	\begin{align*}
	\lambda_2(\Gamma_{n,k}) =  \frac{k-1}{n-1}\binom{n}{k}(n-k-1)!.
	\end{align*}\label{conj:main-conjecture}
\end{conj}

Our next result is about Conjecture~\ref{conj:main-conjecture}, for the $5$-cycles.
\begin{thm}
	For any $n\geq 7$, $\lambda_2(\Gamma_{n,n-5}) = \frac{n(n-2)(n-3)(n-4)(n-6)}{5}$. Moreover, this eigenvalue is afforded by the irreducible character corresponding to the partition $[n-1,1]$.\label{thm:main2}
\end{thm}

This paper is organized as follows. In Section~\ref{sect:proof1}, we prove Theorem~\ref{thm:main}. In our proof of Theorem~\ref{thm:main}, we use the representation theory of the symmetric group. In particular, our proof relies on the recursive Murnaghan-Nakayama rule (see \cite{sagan2001symmetric}). In Section~\ref{sect:proof2}, we show that our method can also be applied to retrieve the result of Siemons and Zalesski in Theorem~\ref{thm:large}. Finally, we prove Theorem~\ref{thm:main2} in Section~\ref{sect:5-cycles}. 

\section{Proof of the main results}
\subsection{Proof of Theorem~\ref{thm:main}}\label{sect:proof1}
Throughout this section, we let $n,k\in \mathbb{N}$ such that $2\leq  k \leq \min \left(n,2\log_{\frac{k}{e}} \left(\frac{n(n-2)}{2e}\right)-1\right)$. 
It is also assumed that the reader is familiar with  the notion of Specht modules \cite[Section~2.3]{sagan2001symmetric}, the Hook Length Formula \cite[Section~3.10]{sagan2001symmetric} and the Murnaghan-Nakayama Rule \cite[Section~4.10]{sagan2001symmetric}. 

We note that since the connection set of $\Gamma_{n,k}$ is a conjugacy class, the graph $\Gamma_{n,k}$ is a normal Cayley graph. It is well-known that the spectrum of a normal Cayley graph on a group $G$ can be determined by the irreducible characters of $G$ (see \cite{babai1979spectra} and \cite[Theorem~2.5]{liu2018eigenvalues} for more details).

Recall that a partition of the integer $n$ is a non-increasing sequence of positive integers summing to $n$. If $\lambda = [n_1,n_2,\ldots,n_t]$ is a partition of $n$, then we write $\lambda \vdash n$. 
It is well-known that there is a bijective correspondence between the irreducible representations of $\sym(n)$ and the set of all partitions of $n$. Each  $\lambda \vdash n$ corresponds to an irreducible $\mathbb{C}\sym(n)$-module $S^\lambda$ (i.e., a complex irreducible representation of $\sym(n)$) called the $\lambda$-\emph{Specht} module. Let $\chi^{\lambda}$ be the character afforded by $S^\lambda$. For any $\lambda \vdash n$, we denote the dimension of $S^\lambda$ by $f^\lambda$ (i.e., $f^\lambda = \chi^\lambda (id)$). If $\lambda\vdash n$ and $\sigma\in \sym(n)$ has cycle type $\tau$, then we define $\chi^{\lambda}_\tau := \chi^\lambda(\sigma)$. Note that since $C(n,k)$ is the conjugacy class of the $(n-k)$-cycles, the cycle type of a permutation of $C(n,k)$ is $(n-k,1^k)$.

Using the famous result of Babai in \cite{babai1979spectra}, the eigenvalues of $\Gamma_{n,k}$ are given in the following lemma.

\begin{lem}
	For any $n,k\in \mathbb{N}$ such that $1\leq k\leq n$, the eigenvalues of $\Gamma_{n,k}$ are the numbers
	\begin{align*}
	\xi_{\chi^\lambda} &= \frac{\chi^\lambda_{(n-k,1^{k})}}{f^\lambda} \binom{n}{k} (n-k-1)!,
	\end{align*}
	for all $\lambda \vdash n$.\label{lem:main1}
\end{lem}
Given two Young diagrams $\lambda$ and $\mu$ such that $\mu \subset \lambda$, the \emph{skew diagram} $\lambda /\mu$ is the set of cells of $\lambda$ that are not in $\mu$.
A \emph{rim hook} $\rho$ of a Young diagram $\lambda \vdash n$ is a skew diagram whose cells are on a path with only upward and rightward steps (see \cite{sagan2001symmetric} for details on this). The length of $\rho$ denoted by $|\rho|$ is the number of its cells, and the \emph{leg-length} $\ell\ell(\rho)$ is the number of rows that $\rho$ spans minus $1$. For any $\lambda \vdash n$, we let ${\sf RH}_{n-k}(\lambda)$ be the set of all rim hooks of length $n-k$ of the partition $\lambda.$ As a rim hook $\rho \in {\sf RH}_{n-k}(\lambda)$ is a skew diagram, removal of the cells of $\lambda$ which are in $\rho$ results in a Young diagram that corresponds to a partition of $k$. We denote the Young diagram obtained from such removal by $\lambda \setminus \rho$.
Using the recursive Murnaghan-Nakayama rule, we obtain the following lemma.

\begin{lem}
	For any $\lambda \vdash n$, the character value of the irreducible character $\chi^\lambda$ is 
	\begin{align*}
		\chi^{\lambda}_{(n-k,1^{k})} &= \sum_{\rho \in {\sf RH}_{n-k}(\lambda)} (-1)^{\ell\ell(\rho)}f^{\lambda \setminus \rho}.
	\end{align*}
\end{lem}
	\begin{proof}
		Let $\rho_1,\rho_2,\dots,\rho_\ell$ be all the distinct rim hooks of length $n-k$ in $\lambda$. By the Murnaghan-Nakayama rule, we have
		\begin{align*}
			\chi^{\lambda}_{(n-k,1^{k})} &= \sum_{i = 1}^{\ell} (-1)^{\ell\ell(\rho_i)} \chi^{\lambda\setminus \rho_i}_{(1^k)} = \sum_{i = 1}^{\ell} (-1)^{\ell\ell(\rho_i)} f^{\lambda\setminus \rho_i}.
		\end{align*}
	\end{proof}

Our strategy is to use the fact that $k$ is relatively small so that there is at most one rim hook of length $n-k$ in any Young diagram $\lambda$. We recall the following lemma which is proved in \cite{behajaina20203}. The proof is  given for completeness.
\begin{lem}
	Let $\lambda = [\lambda_1,\lambda_2,\ldots,\lambda_q] \vdash n$. Assume that $\lambda$ has a rim hook $\rho$ of length $n-k$ and $\mu = [\mu_1,\mu_2,\ldots,\mu_t] \vdash k$ is the Young diagram obtained from $\lambda$ by removing $\rho$. That is, $\mu = \lambda \setminus \rho$.
	
	If $\ell$ is a positive integer such that $2\ell-n \geq k +1$, then $\lambda $ has at most one rim hook of length $\ell$. In particular, if $3k+1<n$, then there is at most one rim hook of length $n-k$ in any Young diagram $\lambda \vdash n$.\label{lem:unique-rim hook}
\end{lem}
\begin{proof}	
	First, we claim that if $\zeta$ is a rim hook of $\lambda$ which is contained in the skew diagram $\lambda/[\lambda_1,1^{q-1}]$, then the length of $\zeta$ is at most $k$. Indeed, since $\zeta$ is contained in $\lambda/[\lambda_1,1^{q-1}]$ and $\mu = \lambda \setminus \rho$, the rim hook $\zeta$ has at most $\mu_1$ columns and $t$ rows. In total, $\zeta$ has at most $\mu_1 + t-1$ cells. Therefore, the length of $\zeta$ is at most $\mu_1 + t-1 \leq k$. 
	
	As a consequence of this result, if a rim hook $\zeta$ has length $k^\prime >$ $k$, then $\zeta$ is not contained in $\lambda/[\lambda_1,1^{q-1}]$. That is, $\zeta$ contains a cell of $\lambda$ of coordinate $(.,1)$ or $(1,.)$. If $\zeta$ contains a cell of coordinate $(.,1)$, then it must start from the cell of coordinate $(q,1)$. In other words, $\zeta$ is the unique rim hook of length $k^\prime$ starting at the cell of coordinate  $(q,1)$. Similarly, if $\zeta$ contains a cell of coordinate $(1,.)$, then $\zeta$ must be the unique rim hook of length $k^\prime$ ending at the cell of coordinate $(1,\lambda_1)$.
	
	Finally, we prove that there is a unique rim hook of length $\ell$ in $\lambda$. Let $\gamma$ and $\gamma^\prime$ be two rim hooks in $\mathsf{RH}_{\ell}(\lambda)$. It is easy to see that $\gamma$ and $\gamma^\prime$ have common cells, otherwise we would have $2\ell \leq n$, which contradicts the fact that $2\ell-n \geq k+1$. Consider the set of common cells $\gamma \cap \gamma^\prime$ of $\gamma$ and $\gamma^\prime$. By the claim that we proved above, the part of the rim hook $\gamma \cap \gamma^\prime$ which is contained in $\lambda/[\lambda_1,1^{q-1}]$ has length at most $k$. Since the length of the union of $\gamma$ and $\gamma^\prime$ is $2\ell - |\gamma\cap\gamma^\prime| = |\gamma| + |\gamma^\prime| - |\gamma\cap\gamma^\prime| \leq n$, we conclude that $2\ell - n \leq |\gamma \cap \gamma^\prime|$. Combining this with the hypothesis that $2\ell -n\geq k+1 $, we have $|\gamma \cap \gamma^\prime|\geq k+1$. In other words, $\gamma \mbox{ and }\gamma^\prime$ has a common cell of coordinate $(.,1)$ or $(1,.)$. It follows that $\gamma = \gamma^\prime$. The second statement of Lemma~\ref{lem:unique-rim hook} is obtained by taking $\ell=n-k$.
\end{proof}

By Lemma~\ref{lem:unique-rim hook}, there is at most one rim hook of length $k$ in $\lambda\vdash n$, whenever $2\leq k \leq \left\lfloor \frac{n-1}{3}  \right\rfloor$. Since $k$ belongs to this interval (by the hypothesis of Theorem~\ref{thm:main}), the non-zero eigenvalues of $\Gamma_{n,k}$ are of the form
\begin{align}
	\xi_{\chi^\lambda} &= (-1)^{\ell\ell(\rho)}\frac{f^{\lambda\setminus \rho}}{f^\lambda} \binom{n}{k} (n-k-1)!,\label{ex:evalues}
\end{align}
	where $\lambda \vdash n$ and $\rho$ is the unique rim hook of length $n-k$ of $\lambda$. We prove that the maximum of $\left\{\xi_{\chi^\lambda}\right\}_{\lambda \vdash n, \lambda \neq [n],[1^n]}$ is attained by the partition $[n-1,1]$. The transpose $[2,1^{n-2}]$ of $[n-1,1]$ also gives the second eigenvalue of $\Gamma_{n,k}$, depending on the parity of $n$ and $k$.
	
	Now, we divide the proof of Theorem~\ref{thm:main} into two cases: the cases $k=2$ and $k\geq 3$.
	\\ \\
	\noindent{\bf a) When $k = 2$}
	\\ \\
	By Lemma~\ref{lem:unique-rim hook}, $\lambda\vdash n$ has at most one rim hook of length $n-2$. If $\lambda\vdash n$ does not have a rim hook of length $n-2$, then the corresponding eigenvalue is equal to $0$. If $\rho$ is the unique rim hook of length $n-2$ of $\lambda \vdash n$, then $\lambda\setminus \rho \in \left\{\ytableausetup{mathmode,boxsize=0.7em}\ydiagram{2},\ytableausetup{mathmode,boxsize=0.7em}\ydiagram{1,1}\right\}$. Hence, the dimension of the $(\lambda\setminus\rho)$-Specht module is always equal to $1$. Consequently, the non-zero eigenvalues of $\Gamma_{n,2}$ are
	\begin{align*}
		\xi_{\chi^\lambda} &=  (-1)^{\ell\ell(\rho)}\frac{1}{f^\lambda} \binom{n}{k} (n-k-1)!,
	\end{align*}
	where $\lambda\vdash n$ has a rim hook of length $n-2$. It is easy to see that the second eigenvalue of $\Gamma_{n,2}$ is obtained from the irreducible characters of smallest dimension which is not equal to $1$. When $n\geq 7$, it is well-known that the only irreducible characters of $\sym(n)$ of degree less than $n$ are those corresponding to the partitions $[n],\ [1^n],\ [n-1,1]$ and $[2,1^{n-2}]$. It is obvious that the partition giving the second eigenvalue (independent of $k$) is $[n-1,1]$ since the leg-length of a rim hook of length $n-2$ on $[n-1,1]$ is equal to $0$. When $3\leq n \leq 6$, we use \verb|Sagemath| \cite{sagemath} to verify that the second eigenvalue is given by $[n-1,1]$. Therefore, $\lambda_2(\Gamma_{n,2}) = \frac{1}{n-1}\binom{n}{2}(n-3)!$.
	\\ \\
	\noindent{\bf b) When $k\geq 3$}
	\\ \\
	First, we recall a basic result about the largest dimension of the irreducible characters of a finite group. If $G$ is a finite group, then we let $b(G)$ be the largest dimension of an irreducible character of $G$. We recall the following well-known result (see \cite{halasi2016largest} for example).

\begin{prop}[\cite{halasi2016largest}]
	If $G$ is a finite group, then $b(G) \leq \sqrt{|G|}$.\label{prop:max-degree}
\end{prop}

Next, we present a lemma on the low dimensional irreducible characters of $\sym(n)$. 
\begin{lem}\label{lem:irrecharless}
	Let $n \geq 19$. If $\phi$ is a character of $\mathrm{Sym}(n)$ of degree less than $3\binom{n}{3}$ and $\chi^{\lambda}$ a constituent of $\phi$, then $\lambda$ is one of  $[n],[1^n],[n-1,1],[2,1^{n-2}],[n-2,2],[2^2,1^{n-4}],[n-2,1^2],[3,1^{n-3}],[n-3,3],[2^3,1^{n-6}],[n-3,1^3],[4,1^{n-4}]$ $, [n-3,2,1]$, $\mbox{or } [3,2,1^{n-5}].$
\end{lem}
A proof of Lemma~\ref{lem:irrecharless} can be found in \cite[Lemma~3.4 and Remark~3.5]{behajaina20203}. 

Suppose that $n\geq 19$. It is easy to see that \eqref{ex:evalues} depends only on $(-1)^{\ell\ell(\rho)}\frac{f^{\lambda\setminus \rho}}{f^\lambda}$. Using Proposition~\ref{prop:max-degree}, if $f^\lambda > 3\binom{n}{3}$, then 
\begin{align*}
	\frac{f^{\lambda\setminus \rho}}{f^\lambda} \leq \frac{b(\sym(k))}{3\binom{n}{3}}\leq \frac{\sqrt{k!}}{3\binom{n}{3}}.
\end{align*}
Using the classical bound $\frac{k^k}{e^{k-1}} \leq k! \leq \frac{k^{k+1}}{e^{k-1}}$, it is easy to verify that when\\ $3\leq k \leq \min \left(n,2\log_{\frac{k}{e}} \left(\frac{n(n-2)}{2e} \right)-1\right)$, we have $\sqrt{k!} \leq \frac{3}{n-1}\binom{n}{3}$. Consequently, if $f^\lambda > 3\binom{n}{3}$, then
\begin{align}
	0 < \frac{f^{\lambda\setminus \rho}}{f^\lambda} \leq \frac{1}{n-1}. \label{eq:to-conj}
\end{align}

Now, we compute the eigenvalues that correspond to the irreducible characters of $\sym(n)$ with dimension less than $3\binom{n}{3}$. The trivial character always affords the valency of $\Gamma_{n,k}$, which is $c_{n,k} :=  \binom{n}{k} (n-k-1)!$. We note that if $\lambda\vdash n$ and $\lambda^\prime$ is the transpose of $\lambda$, then $\chi^{\lambda^\prime}(\sigma) = \chi^{[1^n]}(\sigma)\chi^\lambda(\sigma)$, for any $\sigma\in \sym(n)$. Therefore, the eigenvalue afforded by the partition $[1^n]$ is either equal to $c_{n,k}$ or $-c_{n,k}$, in which case $\Gamma_{n,k}$ is bipartite. The eigenvalue corresponding to $[1^n]$ is precisely $(-1)^{n-k-1}c_{n,k}$. The eigenvalues  of $\Gamma_{n,k}$ corresponding to the other irreducible characters of dimension less than $3\binom{n}{3}$ are given in Table~\ref{table:eigenvalues}. These values were computed using the Murnaghan-Nakayama rule and the Hook Length Formula.
\begin{table}[H]
	\begin{tabular}{|c|c|}
		\hline
		Partition & Eigenvalue\\
		\hline
		\hline
		$[n-1,1]$ & $\frac{k-1}{n-1} c_{n,k}$\\
		\hline
		$[2,1^{n-2}]$ & $(-1)^{n-k+1}\frac{k-1}{n-1} c_{n,k}$\\
		\hline
		$[n-2,2]$ & $
		\frac{\binom{k}{2} - k}{\binom{n}{2}-n} c_{n,k}
		$\\
		\hline
		$[2^2,1^{n-4}]$ & $(-1)^{n-k+1}\frac{\binom{k}{2} - k}{\binom{n}{2}-n} c_{n,k}$
		 \\
		\hline
		$[n-2,1^2]$ & 
		$\frac{\binom{k-1}{2}}{\binom{n-1}{2}}c_{n,k}$\\
		\hline
		$[3,1^{n-3}]$ & $(-1)^{n-k+1}\frac{\binom{k-1}{2}}{\binom{n-1}{2}}c_{n,k}$\\
		\hline
		$[n-3,3]$ & 
		$\frac{\binom{k}{3}-\binom{k}{2}}{\binom{n}{3}- \binom{n}{2}}c_{n,k}$
		\\
		\hline
		$[2^3,1^{n-6}]$ & $(-1)^{n-k+1}\frac{\binom{k}{3}-\binom{k}{2}}{\binom{n}{3}- \binom{n}{2}}c_{n,k}$ \\
		\hline
		$[n-3,1^{3}]$ & 
		$\frac{\binom{k-1}{3}}{\binom{n-1}{3}}c_{n,k}$
		\\
		\hline
		$[4,1^{n-4}]$ & $(-1)^{n-k+1}\frac{\binom{k-1}{3}}{\binom{n-1}{3}}c_{n,k}$\\
		\hline
		$[n-3,2,1]$ &  $\frac{k(k-2)(k-4)}{n(n-2)(n-4)}c_{n,k}$ \\
		\hline
		$[3,2,1^{n-5}]$ &$(-1)^{n-k+1}\frac{k(k-2)(k-4)}{n(n-2)(n-4)}c_{n,k}$\\
		\hline
	\end{tabular}
	\caption{Eigenvalues $\Gamma_{n,k}$ corresponding to irreducible characters of degree less than $3\binom{n}{3}$.}\label{table:eigenvalues}
\end{table}
An analysis of the eigenvalues in Table~\ref{table:eigenvalues} shows that the largest eigenvalue is $\frac{k-1}{n-1}c_{n,k} = \frac{k-1}{n-1}\binom{n}{k}(n-k-1)!$, when $n\geq 19$ and $3\leq k\leq n-2$. Moreover, this eigenvalue is afforded by the irreducible character of $\sym(n)$ corresponding to the partition $[n-1,1]$. 

We conclude that $\lambda_2(\Gamma_{n,k}) = \frac{k-1}{n-1}c_{n,k}$, when $n\geq 19$ and $3\leq k\leq \min \left(n,2\log_{\frac{k}{e}} \left(\frac{n(n-2)}{2e} \right)-1\right)$. For the cases where $3 \leq n\leq 18$ and $3\leq k\leq n-2$, we use \verb|Sagemath| \cite{sagemath} to verify that the second eigenvalue is given by the partition $[n-1,1]$ and it is equal to $\lambda_2(\Gamma_{n,k}) = \frac{k-1}{n-1}c_{n,k}$. This completes the proof of Theorem~\ref{thm:main}.

\subsection{Proof of Theorem~\ref{thm:large}}\label{sect:proof2}
In this section, we use Table~\ref{table:eigenvalues} to retrieve the results of Siemons and Zalesskki in Theorem~\ref{thm:large}.
\subsubsection*{The case $k=0$}
The eigenvalue afforded by the irreducible character corresponding to $[n-1,1]$ is $\frac{-c_{n,0}}{n-1}<0$. When $n$ is odd, it is easy to see that the largest eigenvalue in Table~\ref{table:eigenvalues} is  
$\binom{n-1}{2}^{-1}c_{n,0} = 2(n-3)!$, which is afforded by  $[n-2,1^2]$. When $n$ is even, the largest eigenvalue is $\frac{c_{n,0}}{n-1} = (n-2)!$, which is afforded by $[2,1^{n-2}]$.
\subsubsection*{The case $k=1$}
The eigenvalues afforded by $[n-1,1]$ and $[2,1^{n-2}]$ are both equal to $0$ in this case. When $n$ is odd, the second eigenvalue is afforded by $[2^2,1^{n-4}]$ and is equal to $(-1)^{n}\frac{-1}{\binom{n}{2}-n} c_{n,1} = 2(n-2)(n-4)!$. When $n$ is even, the eigenvalue afforded by $[2^2,1^{n-4}]$ is negative and the second eigenvalue of $\Gamma_{n,k}$ is equal to $\frac{(-1)(-3)}{n(n-2)(n-4)}c_{n,1} = 3(n-3)(n-5)!$, which is afforded by $[n-3,2,1]$.

\section{Second eigenvalue of $\Gamma_{n,n-5}$}\label{sect:5-cycles}
In this section, we prove Theorem~\ref{thm:main2}.
\subsection{Equitable partitions}
In this subsection, we show that the eigenvalue afforded by the irreducible character corresponding to the partition $[n-1,1]$ appears as an eigenvalue of an equitable partition of $\Gamma_{n,k}$.

The following lemma is straightforward and is given without a proof (see \cite[Lemma~5]{Cioaba2019second}).

\begin{lem}
	Let $G$ be a finite group and $H\leq G$. If $X = \operatorname{Cay}(G,C)$ is a Cayley graph of $G$, then the partition of $X$ into left cosets of $H$ is an equitable partition of $X$.\label{lem:equitable-subgroup}
\end{lem}

Let $G := \sym(n)$ and for any $i\in \{1,2,\ldots,n\}$, let $G_i$ be the stabilizer of $i$ in the natural action of $G$ on the set $\{1,2,\ldots,n\}$. For any $s,t \in \{1,2,\ldots,n\}$, we define $G_{t,s} := \{ \sigma \in G \mid \sigma(t) =s \}$. By Lemma~\ref{lem:equitable-subgroup}, the partition $G/G_i$ is equitable for any $i\in \{1,2,\ldots,n\}$. If $B_{\Pi_i} = (b_{s,t})_{s,t \in \{1,2,\ldots,n\} }$ is the quotient matrix corresponding to the equitable partition $\Pi_i$ of $\Gamma_{n,k}$ given by $G/G_i$, then by \cite{Cioaba2019second}
\begin{align*}
	b_{s,t} = |C(n,k)\cap G_{t,s}|=
	\left\{
		\begin{aligned}
			&\binom{n-1}{n-k}(n-k-1)! \ \ \ \ \  &\mbox{ if } t=s\\
			&\binom{n-2}{n-k-2}(n-k-2)! \ \ \ \ \  &\mbox{ if } t\neq s.
		\end{aligned}
	\right.
\end{align*}
It is not hard to see that 
\begin{align*}
	B_{\Pi_i} &= \binom{n-2}{n-k-2}(n-k-2)!(J-I) + \binom{n-1}{n-k}(n-k-1)! I,
\end{align*}
where $J$ is the $n\times n$ all ones matrix and $I$ is the $n\times n$ identity matrix. The eigenvalues of $B_{\Pi_i}$ are
\begin{align*}
	\binom{n}{k}(n-k-1)! \mbox{ and } \frac{k-1}{n-1}\binom{n}{k}(n-k-1)!,
\end{align*}
which are the eigenvalues afforded by the irreducible characters corresponding to $[n]$ and $[n-1,1]$, respectively.

\subsection{A recursive method}
In \cite{Cioaba2019second}, Huang et al. gave a recursive method to compute the second eigenvalue of highly transitive groups.

Throughout this subsection, we let $G\leq \sym(\Omega)$ be a finite transitive group acting on $\Omega = \{1,2,\ldots,n\}$. Let $G^{(0)} =G$ and for any $k\geq 1$, let 
\begin{align*}
	G^{(k)} &= G_{n} \cap G_{n-1} \cap \ldots \cap G_{n-k+1}.
\end{align*} Let $T$ be a union of conjugacy classes of $G$.
Define 
\begin{align*}
	\left\{
		\begin{aligned}
			T_k &= T_{k-1} \setminus \left( T_{k-1} \cap G_k \right), \mbox{ for any  }k\geq 1,\\
			T_0 &= T.
		\end{aligned}
	\right.
\end{align*}
In other words, if $\operatorname{Supp(\sigma)} = \{ i\in \Omega \mid \sigma(i) \neq i \}$, then $T_k = \{ \sigma \in T \mid \{1,2,\ldots,k\}\subset \operatorname{supp}(\sigma)  \}$. For any $i\geq 0$ and $k\geq 0$, let
\begin{align}
	X_{k,i} &= \cay(G^{(i)},T_k\cap G^{(i)}).\label{eq:graph-smaller-recursion}
\end{align}

Let $\Pi$ be the partition of $G$ into left cosets of any stabilizer of a point of $G$. From the results in the previous subsection, we know that this partition is equitable. Let $B_\Pi$ be the quotient matrix corresponding to this equitable partition. By Lemma~\ref{lem:equitable-subgroup}, the partition of $G^{(i)}$ into left cosets of any of its point-stabilizers is an equitable partition of $X_{k,i}$, for any $k\geq 0$. Let $B_\Pi^{(k,i)}$ be the quotient matrix of this equitable partition. The second eigenvalue $\lambda_2\left(B_\Pi^{(k,i)}\right)$ of $B^{(k,i)}_\Pi$ was computed in \cite{Cioaba2019second} and it is equal to
\begin{align*}
	\lambda_2\left(B^{(k,i)}_\Pi\right) &= |T_k\cap G^{(i)} \cap G_{k+1}| - |T_k\cap G^{(i)}\cap G_{k+2,k+1}|.
\end{align*}

When $G$ is highly transitive, the second eigenvalue of the normal Cayley graph $\cay(G,T)$ can be computed via a recursive method on the graphs defined in \eqref{eq:graph-smaller-recursion}. 
\begin{lem} \cite[Theorem~14]{Cioaba2019second}
	Let $m = \displaystyle\max_{\sigma \in T}|\operatorname{supp}(\sigma)|$. If $G$ is $(m+a)$-transitive for some $a\geq 1$ and $\lambda_2(X_{k,a-1}) = \lambda_2\left(B_\Pi^{(k,a-1)}\right)$ for any $0\leq k\leq m-1$, then $$\lambda_2(\cay(G,T)) = \lambda_2(X_{0,0}) = \lambda_2(B_\Pi).$$\label{lem:recursive}
\end{lem}

In the next subsection, we find the second eigenvalue of $\Gamma_{n,n-5}$ using this recursive method. 

\subsection{Proof of Theorem~\ref{thm:main2}}
Let $T = C(n,n-5)$ be the conjugacy class of $5$-cycles of $\sym(n)$. Our main tool to prove Theorem~\ref{thm:main2} is Lemma~\ref{lem:recursive}. Since $\langle T\rangle = \alt(n)$, the graph $\Gamma_{n,n-5}$ is disconnected and is the disjoint union of two copies of $\cay(\alt(n),T)$. Therefore, $\lambda_2(\Gamma_{n,n-5}) = \lambda_2(\cay(\alt(n),T))$. Let us apply Lemma~\ref{lem:recursive} on $X = \cay(\alt(n),T)$.

As the elements of $T$ are $5$-cycles, $m= \displaystyle\max_{\sigma \in T}|\operatorname{supp}(\sigma)| = 5$. Since $\alt(n)$ is $(n-2)$-transitive, it is easy to see that $\alt(n)$ is $(m+a)$-transitive, whenever $a \leq n-7$. Let $a = n-7$ and $G = \alt(n)$. Now, it is enough to verify that 
\begin{align}
	\lambda_2(X_{k,n-8}) = |T_k\cap G^{(n-8)} \cap G_{k+1}| - |T_k\cap G^{(n-8)}\cap G_{k+2,k+1}|, \label{eq:equality-5-cycle}
\end{align}
for any $0\leq k \leq 4$. 

Note that $G^{(n-8)} = \alt(8)$ and for any $0\leq k \leq 4$, $X_{k,n-8} = \cay(\alt(8), T_k \cap \alt(8))$. Using \verb*|Sagemath| \cite{sagemath} and \verb*|SciPy| \cite{virtanen2020scipy}, we were able to verify \eqref{eq:equality-5-cycle} for any $0\leq k\leq 4$. We have compiled in the following table the values of $\lambda_2(X_{k,n-8})$ and $\lambda_1(X_{k,n-8})$.
\begin{table}[H]
	\begin{tabular}{|c|c|c|}
		\hline
		$k$ & $\lambda_1(\cay(\alt(8), T_k \cap \alt(8))) $ & $\lambda_2(\cay(\alt(8), T_k \cap \alt(8)))$ \\
		\hline\hline
		$0$ & $1344$ & $384$ \\
		$1$ & $840$ & $300$ \\
		$2$ & $480$ & $216$  \\
		$3$ & $240$ & $138$ \\
		$4$ & $96$ & $72$ \\
		\hline 
	\end{tabular}
\end{table}
By Lemma~\ref{lem:recursive}, we conclude that $\lambda_2(\Gamma_{n,n-5}) = \frac{n-6}{n-1} \binom{n}{5} 4! = \frac{n(n-2)(n-3)(n-4)(n-6)}{5}$.


\end{document}